\documentclass[11pt]{article}

\usepackage[margin=1.6in]{geometry}

\usepackage{amsmath}
\usepackage{amsthm}
\usepackage{amssymb}
\usepackage{hyperref} 
\usepackage[hypcap]{caption}

\newtheorem{mainthm}{Theorem}

\newtheorem{thm}{Theorem}[section]

\newtheorem{lem}[thm]{Lemma}

\newtheorem*{thmnon}{Theorem}

\newtheorem*{lemnon}{Lemma}

\theoremstyle{definition}
\newtheorem{defn}{Definition}

\theoremstyle{remark}

\newtheorem{claim}{Claim}

\newtheorem*{remark}{Remark} 
\newtheorem*{claimnon}{Claim}

\newcommand{\QQ}{\mathbb Q}
\newcommand{\ZZ}{\mathbb Z}

\newcommand{\CC}{\mathbb C}

\newcommand{\Mm}{\mathcal M}

\newcommand{\inv}{^{-1}}

\DeclareMathOperator{\Norm}{Norm}  
\DeclareMathOperator{\GF}{GF}  
\DeclareMathOperator{\Gal}{Gal}  
\DeclareMathOperator{\Fix}{Fix}

\newcommand{\ideal}[1]{\langle #1 \rangle}

\newcommand{\floor}[1]{\left\lfloor  #1 \right\rfloor} 
\newcommand{\comout}[1]{}

\newcommand{\ab}{\allowbreak}

\title{Effective versions of two theorems of Rado}

\author{
Jason Bell\thanks{Department of Pure Mathematics, University of Waterloo, Canada. jpbell@uwaterloo.ca}
\and
Daryl Funk\thanks{School of Mathematics and Statistics, Victoria University of Wellington, New Zealand. Current affiliation: Dept.\ of Mathematics, Douglas College, Canada. funkd@douglascollege.ca}
\and
Byoung Du Kim\thanks{School of Mathematics and Statistics, Victoria University of Wellington, New Zealand. byoungdu.kim@vuw.ac.nz, dillon.mayhew@vuw.ac.nz}
\and
Dillon Mayhew${}^\ddagger$
}

\date{October 24, 2019} 

\begin{document}

\maketitle

\begin{abstract}
Let $M$ be a representable matroid on $n$ elements.  
We give bounds, in terms of $n$, on the least positive characteristic and smallest field over which $M$ is representable.  
\end{abstract}

Our starting point is given by the following two theorems of Rado \cite{MR0088459}. 

\begin{mainthm}[Rado, 1957] \label{thm:Rado1}
Let $M$ be a matroid representable over a field $K$.  
Then $M$ is representable over a simple algebraic extension of the prime field of $K$.  
\end{mainthm}

\begin{mainthm}[Rado, 1957] \label{thm:Rado2} 
Let $K$ be an extension field of $\QQ$ of degree $N$, and let $M$ be a matroid representable over $K$.  
Then there is a positive integer $c$ such that given any prime $p > c$ there is a positive integer $k = k(p) \leq N$ such that $M$ is representable over $\GF(p^k)$.  
For infinitely many $p$, $k(p)=1$. 
\end{mainthm}

Together, these two theorems say that if a matroid is linearly representable, then it is representable over a finite field.  
We ask, given a representable matroid on $n$ elements, how large must such a field be? 
That is, given an $n$-element representable matroid $M$, what bound, depending just on $n$, can we place on the size of a field required to represent $M$?  

To that end, let $\Mm_n$ be the set of all representable matroids on $n$ elements.  
For a matroid $M$, let $c(M)$ be the least positive characteristic of a field over which $M$ is representable.  
For each positive integer $n$, define 
\[  
c(n) = \max \{ c(M) : M \in \Mm_n \} . 
\]
Let $f(M)$ be the order of the smallest field over which $M$ is representable.  
For each positive integer $n$, define 
\[ 
f(n) = \max \{ f(M) : M \in \Mm_n \}.
\]

By Rado's Theorems \ref{thm:Rado1} and \ref{thm:Rado2} above, $c(n)$ exists and $f(n)$ is finite for all $n$.  
Note that $c(n) \leq f(n)$ for all $n$, and that, since adding a loop to an $n$-element matroid yields a matroid on $n+1$ elements representable over exactly the same fields, $c$ and $f$ are non-decreasing.  
A result of Brylawski \cite{MR664703} provides a lower bound for $c$ (and thus for $f$; see Section \ref{sec:lower_bound}). 
We ask for upper bounds on $c(n)$ and $f(n)$.  
For matroids on at most 8 elements, Table \ref{table:eightelements} summarises the data (the fact that $f(8)=11$ is courtesy G.\ Royle [personal communication]. 
\begin{table}[tbp]
\begin{center}
 \begin{tabular}{ c  c   c }
  $n$ & $c(n)$ & $f(n)$ \\   
  \hline
  \hline
  1 & 2 & 2 \\
  \hline
  2 & 2 & 2 \\
  \hline
  3 & 2 & 2 \\ 
  \hline
  4 & 2 & 3 \\ 
  \hline
  5 & 2 & 4 \\ 
  \hline
  6 & 2 & 5 \\ 
  \hline
  7 & 3 & 7 \\ 
  \hline
  8 & ? & 11  \\
  \hline  
\end{tabular}
\end{center}
\caption{}
\label{table:eightelements}
\end{table} 

We obtain the following bounds.  

\begin{mainthm} \label{thm:main}
For all positive integers $n$, 
\[
\log_2 \log_2 c(n) \leq n^5 \ \text{ and } \ \log_2 \log_2 \log_2 f(n) \leq n^3 .
\]
\end{mainthm}

The following fact falls out of the proof of Theorem \ref{thm:main}. 

\begin{mainthm} \label{thm:boundforGFp}
Let $M$ be an $n$-element matroid representable over a field of characteristic $0$, and let $p$ be a prime satisfying 
\[ \log_2 \log_2 \log_2 p > n^5 . \]
Then $M$ is representable over $\GF(p)$. 
\end{mainthm}

We consider the cases of representability over only positive characteristic (Theorem \ref{thm:bounds_not_char_zero}) and representability over characteristic $0$ (Theorem \ref{thm:bounds_char_zero}) separately. 
Theorem \ref{thm:main} then follows immediately from these results. 

By Table \ref{table:eightelements}, we may assume throughout the rest of the paper that $n > 7$. 

\section{Bounding the degree of a field extension} 
\label{sec:bounding_degree_field_extension}

Our first step is to prove an effective version of Rado's Theorem \ref{thm:Rado1}:  

\begin{thm} \label{thm:EffectiveRado1} 
Let $M$ be a matroid on $n$ elements representable over a field $K$.  
Then $M$ is representable over a simple algebraic extension of the prime field of $K$ of degree at most 
$2^{2^{2n^2}}$.
\end{thm}

\subsection{A system of polynomials arising from a matroid} \label{sec:polynomials_from_a_matroid}
Our approach is a standard one in studies of representability of matroids over fields. 
Indeed, it is that used by Rado in \cite{MR0088459}; however, as Rado's proofs are non-constructive, beyond this starting point our proofs require substantially more work than Rado's. 
We assign to an $n$-element, rank-$r$ matroid $M$ an $r \times n$ matrix $A$ whose entries are indeterminates $x_1, \ldots, x_t$, where $t = rn$. 
Each element of the matroid is represented by a column of the matrix. 
From this matrix we obtain a system of polynomial equations in $\ZZ[x_1, \ldots, x_t]$ as follows. 
For each $r$-element subset $X$ of the ground set of $M$, there is a corresponding $r \times r$ submatrix of $A$ whose columns are those representing the elements in $X$. 
Setting the determinants of $r \times r$ submatrices corresponding to dependent sets to zero, and demanding that the determinants of those $r \times r$ submatrices that correspond to bases be nonzero, yields a system of polynomials. 
The latter conditions may be expressed by multiplying each polynomial $f_i$ obtained from a basis by a new dummy variable $z_i$ and subtracting 1 to form the polynomial equation $z_i f_i - 1 = 0$.  
Alternatively, these conditions may be expressed by the single polynomial obtained by taking the product of all determinants corresponding to bases, then multiplying by a single dummy variable and subtracting 1.  
Writing $f_i$ for the polynomials obtained by taking the $r \times r$ determinants of $A$, and $B$ for the index set of determinants given by $r \times r$ submatrices whose columns correspond to bases of $M$, this gives the equation $z \prod_{i \in B} f_i -1 = 0$.  
This is more expensive in terms of the degree of the resulting polynomial, but cheaper in terms of the number of new variables added to the system.  
We therefore prefer this second formulation. 
In either case, the system can be interpreted in any field $K$ by extending the canonical homomorphism $\ZZ \to K$ to a map $\ZZ[x_1, \ldots, x_t] \to K[x_1, \ldots, x_t]$ in the natural way.  
Those fields over which $M$ is representable are exactly the fields over which the corresponding system of polynomials has a solution.  

Given a system of polynomials $f_1, \ldots, f_s \in \ZZ[x_1, \ldots, x_t]$ arising in this way from a rank-$r$, $n$-element matroid, we will require bounds on four parameters, described in the following lemma. 
Let $\deg f$ denote the total degree of the polynomial $f$; set $d = \max_i \deg f_i$. 
The \emph{height} $H(f)$ of a polynomial $f$ is the maximum absolute value of a coefficient in $f$; set $H = \max_i H(f_i)$. 

\begin{lem} \label{lem:boundsonsystemparameters} 
Let $f_1, \ldots, f_s \in \ZZ[x_1, \ldots, x_t]$ be a system of polynomials arising as described above from a rank-$r$, $n$-element matroid. 
Then 
$s \leq 2^n$, 
$t \leq n^2 + 1$, 
$d \leq n 2^n$, and 
$H \leq n^{n2^n}$. 
\end{lem}

\begin{proof} 
It is straightforward to see that 
$s \leq {n \choose r} \leq 2^n$, 
$t \leq nr + 1 \leq n^2 + 1$, and 
$d = r \cdot {n \choose r} + 1 \leq n 2^n$.  
A bound on $H$ is less obvious, but no more difficult. 
Since the polynomials in our system corresponding to non-bases have height one, the maximum height of a polynomial in our system will be that of the polynomial obtained by taking the product of all $r \times r$ determinants corresponding to bases of $M$. 
Since this polynomial is obtained as the product of at most ${n \choose r} \leq 2^n$ polynomials given by determinants, each of which has $r! < n^n$ terms, the number of terms in the product, before summing identical monomials, is at most $(n^n)^{2^n}$. 
Hence the height of this polynomial is certainly at most $n^{n2^n}$. 
Thus for our system, $H \leq n^{n2^n}$. 
\end{proof}

\subsection{Algebraic tools} 
Before proceeding, we summarise the algebraic notions we require.  
A system of polynomials $f_1, \ldots, f_s \in K[x_1, \ldots, x_t]$ is \emph{consistent} if it has a solution in the algebraic closure $\overline{K}$ of $K$; 
that is, there is an assignment of values $x_i = \alpha_i \in \overline{K}$, for $i \in \{1, \ldots, t\}$, so that for each $j \in \{1, \ldots, s\}$, $f_j(\alpha_1, \ldots, \alpha_t)=0$.  
By Hilbert's Nullstellensatz, a system of polynomials $P$ in the ring of polynomials $K[x_1, \ldots, x_t]$ is consistent if and only if the ideal generated by $P$ in $K[x_1, \ldots, x_t]$ does not contain 1 
(one reference is \cite[Chapter 30]{MR1276273}).  

Given a field extension $L \supseteq K$, $L$ can be viewed as a vector space $V$ over $K$.  
The \emph{degree} of the extension is the dimension of this vector space, denoted $[L:K]$.  
Given an element $\alpha \in L$, the map $m_\alpha \colon L \to L$ defined by multiplication by $\alpha$ is an $K$-linear transformation.  
When $[L:K]$ is finite, the map $m_\alpha$ is given by a matrix, with respect to a chosen basis for $V$; different bases yield different but similar matrices for $m_\alpha$.  
The \emph{norm} of $\alpha$, denoted $\Norm_{L/K} \alpha$, is the determinant of a matrix corresponding to the linear transformation $m_\alpha$. 
The norm is a map $L \to K$ satisfying $\Norm_{L/K} (\alpha\beta) = (\Norm_{L/K} \alpha)(\Norm_{L/K} \beta)$.  

A nonzero polynomial $f \in K[X]$ is said to \emph{split} in $K$ if each of its irreducible factors has degree 1.  
A \emph{splitting field} for a polynomial $f \in K[X]$ of degree $d$, is a field extension $L$ of $K$, in which $f$ splits
\[
f(x) = a \prod_{i=1}^d (x-\alpha_i) 
\]
for some $a \in K$, such that $L$ is generated over $K$ by the roots $\alpha_i \in L$ of $f$.  

\begin{lem}[\cite{MR1276273}, Theorem 17.18, Lemma 17.20, Corollary 17.21] \label{lem:splitting_fields} 
Let $f \in K[X]$ be a nonzero polynomial. 
There exists a field $L \supseteq K$ such that $f$ splits over $L$, and 
$L$ contains a unique splitting field $L$ for $f$ over $K$.  
\end{lem} 

A polynomial $f \in K[X]$ of degree $d$ has \emph{distinct roots} if $f$ has $d$ different roots in every splitting field $L \supseteq K$ for $f$.  
A nonzero polynomial $f \in K[X]$ is \emph{separable} over $K$ if each irreducible factor of $f$ in $K[X]$ has distinct roots; otherwise $f$ is \emph{inseparable}.  

For any field extension $K \subseteq L$, the \emph{Galois group} $\Gal(L/K)$ of $L$ over $K$ is the subgroup of the group of automorphisms of $L$ consisting of those automorphisms that fix all elements of $K$.  
Given an arbitrary subgroup $H$ of the group of automorphisms of $L$, define $\Fix(H) = \{\alpha \in L : \sigma(\alpha) = \alpha$ for all $\sigma \in H\}$. 
Then $\Fix(H)$ is a subfield of $L$. 
A field extension $L \supseteq K$ is \emph{Galois} if $[L:K]$ is finite and $K=\Fix(\Gal(L/K))$.  

\begin{lem}[\cite{MR1276273}, Theorem 18.13] \label{lem:splitting_iff_Galois} 
Let $L \supseteq K$ be a field extension of finite degree.  
The following are equivalent.  
\begin{enumerate}
\item  $L$ is a splitting field over $K$ for some separable polynomial over $K$.  
\item $L$ is a Galois extension of $K$. 
\end{enumerate}
\end{lem} 

\begin{lem}[\cite{MR1276273}, Lemmas 18.3, 18.19, Corollary 23.10] \label{lem:NorminGaloisext}
Let $L \supseteq K$ be a Galois extension, and let $G$ be the Galois group of $L$ over $K$.  
Let $f \in K[X]$ be nonzero, and let $\Omega = \{\alpha \in L : f(\alpha)=0\}$ be nonempty.
Then 
\begin{enumerate}
\item $|G| = [L:K]$.  
\item The action of $G$ on $L$ permutes the elements of $\Omega$. 
\item If $f$ is irreducible and $L$ is a splitting field over $K$ for some polynomial in $K[X]$, then $G$ acts transitively on $\Omega$.    
\item  For $\alpha \in L$, 
\[
\Norm_{L/K} \alpha = \prod_{\sigma \in G} \sigma(\alpha) .
\]
\end{enumerate}
\end{lem}

We also use Gauss's Lemma: 

\begin{lemnon}[Gauss's Lemma; \cite{MR1276273}, Lemma 16.19] 
Let $R$ be a unique factorisation domain and $K$ its field of fractions.  
A nonzero polynomial in $R[X]$ is irreducible in $R[X]$ if and only if it is irreducible in $K[X]$. 
\end{lemnon} 

Let $f_1, \ldots, f_s \in K[x_1, \ldots, x_t]$ be a system of polynomials with coefficients in the field $K$.  
For each index $i \in \{1, \ldots, t\}$, let $\mathbf{x}-i$ denote the set of indeterminates $\{x_1, \ldots, x_t\} \setminus \{x_i\}$.  
For each pair of indices $i, j$, we may regard $f_j$ as a single-variable polynomial in $x_i$ with coefficients in the field $K(\mathbf{x}-i)$.  
By Gauss's Lemma, it is sufficient that $f$ be irreducible in $K[x_1, \ldots, x_t]$ to guarantee that $f$ be irreducible in $K(\mathbf{x}-i)[x_i]$ for any $i$.  

In order to take advantage of the tools of Galois Theory, we will want to select a polynomial $f_j$ from our system that has an irreducible factor with distinct roots, when viewed as a polynomial in $K(\mathbf{x}-i)[x_i]$ for some $i \in \{1, \ldots, t\}$.  
We need to deal with the possibility that every polynomial in our system, when viewed as a polynomial in the polynomial ring $K(\mathbf{x}-i)[x_i]$, for every $i$, is inseparable. 
The following lemma describes the situation in this rather special case. 

\begin{lem}[\cite{MR1276273}, Corollary 19.6] 
\label{lem:if_nonseparable} 
Let $K$ be a field.  
Let $f \in K[X]$ be an irreducible polynomial that does not have distinct roots. 
Then the characteristic of $K$ is a prime $p$ and $f(X) = g(X^p)$ for some irreducible polynomial $g \in K[X]$.
\end{lem}

\subsection{Reduced systems of polynomials}

We need one more notion before 
proving Theorem \ref{thm:EffectiveRado1}. 
The \emph{variety} defined by the polynomials $f_1, \ldots, f_s \in K[x_1, \ldots, x_t]$ is the set of all tuples $(\gamma_1, \ldots, \gamma_t) \in \overline{K}^t$ that are solutions to the system $f_1 = 0, \ldots, f_s = 0$, and is denoted $V(f_1, \ldots, f_s)$.  
Denote by $\deg(f, x)$ the degree of the polynomial $f$ in indeterminate $x$.  
Let $S = \{f_1, \ldots, f_s\}$ be a system of polynomials in indeterminates $x_1, \ldots, x_t$ with coefficients in the field $K$.  
The \emph{leading indeterminate} of $S$ is the unique indeterminate (among those appearing in a term with nonzero coefficient) $x_l$ satisfying: 
\begin{itemize} 
\item for some polynomial $f \in S$, $\deg(f,x_l)>0$; 
\item for all polynomials $f \in S$, and for all $i>l$, $\deg(f,x_i)=0$.  
\end{itemize} 
Write each polynomial $f \in S$ as a sum of monomials each consisting of a single power $x_l^n$ of the leading indeterminate $x_l$ of the system, together with a coefficient $a_n \in K[x_1, \ldots, x_{l-1}]$, where each power of $x_l$ appears in no more than one term; that is, write 
$f = a_d x_l^d + a_{d-1} x_l^{d-1} + \cdots + a_1 x_l + a_0$.   
The \emph{leading coefficient} of $f$ is the coefficient $a_d \in K[x_1, \ldots, x_{l-1}]$ of its highest power $x_l^d$ of the leading indeterminate $x_l$ of the system, where both $d$ and $a_d$ are nonzero.  
Thus a polynomial having no term containing the leading indeterminate has no leading coefficient. 

Let $P=\sqrt{\ideal{S}}$ be the radical ideal of the ideal generated by $f_1, \ldots, f_s$ in $K[x_1, \ldots, x_t]$. 
The system $f_1, \ldots, f_s \in K[x_1, \ldots, x_t]$ is \emph{reduced} over $K$ if 
\begin{itemize} 
\item each of $f_1, \ldots, f_s$ is irreducible, 
\item $x_t$ is the leading indeterminate of the system, 
\item no leading coefficient is in $P$. 
\end{itemize}

{
These may be thought of as non-degeneracy conditions that we wish to impose on our system of polynomials: 
If $f \in S$ is reducible, then whenever $f(\gamma_1, \ldots, \gamma_t) =0$ one of its irreducible factors must be zero; choosing such a factor from each polynomial in $S$ yields a simpler system (which, if the original system is consistent, will remain consistent as long as the factors are chosen appropriately). 
And clearly there is no reason to work in $K[x_1, \ldots, x_l, \ldots, x_t]$ if indeterminates $x_{l+1}, \ldots, x_{t}$ do not appear in any polynomial in $S$ other than with degree $0$ or in a term whose coefficient is $0$; we may just as well work in $K[x_1, \ldots, x_l]$. 
The third condition is a little more subtle. 
Consider a polynomial in $S$, 
$f = a_d x_t^d + a_{d-1} x_t^{d-1} + \cdots + a_1x_t + a_0$, 
as a polynomial in the indeterminate $x_t$ with coefficients $a_d, \ldots, a_0$ in $K[x_1, \ldots, x_{t-1}]$. 
Write $f = a_d x_t^d + p$, where $p = a_{d-1} x_t^{d-1} + \cdots + a_1x_t + a_0$, and let $\gamma \in V(P)$. 
If $a_d \in P$, then both $a_d$ and $p$ are zero at $\gamma$. 
Thus in the leading term $a_d x_t^d$ of the polynomial $f$, the indeterminate $x_t$ is redundant: removing $f$ from $S$ while adding $a_d$ and $p$ to $S$ yields a simpler system of polynomials. 
This new system has one more polynomial than $S$, but the two polynomials added each have degree strictly smaller than the polynomial $f$ that has been removed. 

There are two main technical reasons that we wish to work with a reduced system, which we summarise in the following sketch of the ideas used in the proof of Theorem \ref{thm:EffectiveRado1}. 
The proofs of Lemmas \ref{lem:mainlemma} and \ref{lem:Norm_is_consistent} provide the details. 

We prove Theorem \ref{thm:EffectiveRado1} inductively, on the number of indeterminates in the system of polynomials $S$ given by a matroid as described in Section \ref{sec:polynomials_from_a_matroid}. 
To do so, we choose a polynomial $f \in S$. 
Considering $f$ as a polynomial in the single indeterminate $x_t$ with coefficients in $K[x_1, \ldots, x_{t-1}]$, we choose a root $x_t = \alpha$ of $f$ in the algebraic closure of the field ${K(x_1, \ldots, x_{t-1})}$. 
So that we may make use of item 3 of Lemma \ref{lem:NorminGaloisext}, we wish $f$ to be irreducible. 
To take advantage of the properties of elementary symmetric polynomials, we form the monic polynomial $f' = (1/a_d)f \in K(x_1, \ldots, x_{t-1})[x_t]$ by dividing $f$ by its leading coefficient $a_d \in K[x_1, \ldots, x_{t-1}]$. 
We make the substitution $x_t = \alpha$ in each of the polynomials $f_j$ in our system $S$, and taking norms we obtain a new system of polynomials in $K[x_1, \ldots, x_{t-1}]$, for which we obtain a solution $x_1 = \gamma_1$, \ldots, $x_{t-1} = \gamma_{t-1}$, each $\gamma_i \in \overline{K}$, via our induction hypothesis. 
We next wish to find a root $x_t = \gamma_t \in \overline{K}$ such that $(\gamma_1, \ldots, \gamma_{t-1}, \gamma_t)$ is a solution to our original system. 
Roughly speaking, because we divided by $a_d$ to make $f$ monic, we must now consider a system of the form 
$\{(a_d)^{m_j} \cdot \Norm_{K_1/K_0} f_j(\alpha) : f_j \in S\}$, 
where $m_j$ is a positive integer, $K_0$ is the field $K(x_1, \ldots, x_{t-1})$, and $K_1$ is the splitting field in $\overline{K_0}$ for $f$ over $K_0$. 
We will wish to use the fact that one of the factors in the expression for the norm given in item 4 of Lemma \ref{lem:NorminGaloisext} must be zero when evaluating at $x_1 = \gamma_1$, \ldots, $x_{t-1} = \gamma_{t-1}$. 
This will be the case provided $a_d$ does not evaluate to zero at $x_1 = \gamma_1$, \ldots, $x_{t-1} = \gamma_{t-1}$. 
Insisting that $a_d \notin P$ is sufficient to guarantee this. 

Fortunately, reduced systems are not hard to find. 
}

{\begin{lem} \label{lem:reduced_system}
Let $h_1, \ldots, h_r \in K[x_1, \ldots, x_u]$ be a consistent system of polynomials, with $\deg(h_j, x_i) \leq D$ for each $j, i$. 
Assume that $x_i = 0$ for each $i \in \{1,\ldots, u\}$ is not a solution of the system. 
Then there is a consistent reduced system of polynomials $f_1, \ldots, f_s \in K[x_{i_1}, \ldots, x_{i_t}]$, where $\{i_1, \ldots, i_t\} \subseteq \{1, \ldots, u\}$, with $\deg(f_j, x_{i_k}) \leq D$ for each $j, i_k$, and with $V(\ideal{f_1, \ldots, f_s}) \subseteq V(\ideal{h_1, \ldots, h_r})$, 
where $\ideal{f_1, \ldots, f_s}$ is generated in $K[x_1, \ldots, x_u]$. 
\end{lem} 
}

\begin{remark} 
In our context, the condition that $x_i = 0$ for each $i$ not be a solution of the system is natural and benign. 
A system of polynomials arising from a matroid as described in Section \ref{sec:polynomials_from_a_matroid} may have the all-zeros solution just in the uninteresting case that the matroid has no bases. 
In this case, every element of the matroid is itself dependent and so the matroid is represented over every field just by a matrix in which every entry is zero. 
But such a matroid has rank zero. 
Since the system of polynomials we construct from a matroid starts with an $r \times n$ matrix of indeterminates, where $r$ is the rank of the matroid, a matroid of rank zero does not even have an associated system of polynomials defined for it. 
Theorems \ref{thm:main} and \ref{thm:boundforGFp} obviously hold for every matrix of rank zero.  
\end{remark}

If $h_1', \ldots, h_r'$ is a system of polynomials chosen so that for each $j \in \{1, \ldots, r\}$, polynomial $h_j'$ is an irreducible factor of $h_j$, and the system $h_1', \ldots, h_r'$ is consistent, then we say $h_1', \ldots, h_r'$ is a \emph{valid choice of factors} of $h_1, \ldots, h_r$. 
Clearly, every consistent system of polynomials has a valid choice of factors.  
Having made a valid choice of factors $h_1', \ldots, h_r'$ from a system of polynomials $h_1, \ldots, h_r \in K[x_1, \ldots, x_u]$, we may consider the ideal $\ideal{h_1', \ldots, h_r'}$ generated in $K[x_1, \ldots, x_u]$ even if $h_1', \ldots, h_r' \in K[x_{i_1}, \ldots, x_{i_t}]$ where $\{{i_1}, \ldots, {i_t}\} \subset \{1, \ldots, u\}$. 
We do so in the following proof. 

{
\begin{proof}[Proof of Lemma \ref{lem:reduced_system}]
Let $S_0 = \{h_1, \ldots, h_r\}$, and let 
$S_1 = \{h_1', \ldots, h_r'\}$ be a valid choice of factors of the polynomials in $S_0$. 
Then $V(\ideal{S_1}) \subseteq V(\ideal{S_0})$. 
Let $x_{i_1}, \ldots, x_{i_l}$ denote the indeterminates with positive degree appearing in a polynomial in $S_1$ in a term with nonzero coefficient, where $x_{i_l}$ is the leading indeterminate of $S_1$. 
If setting all indeterminates appearing in $S_1$ equal to zero were a solution to $S_1$, then setting all of $x_1, \ldots, x_u$ to zero would be a solution to $S_0$. 
Thus $S_1$ does not consist entirely of monomials. 
If no polynomial in $S_1$ has a leading coefficient in $\sqrt{\ideal{S_1}}$, we are done: 
$S_1$ is a reduced system of polynomials in $K[x_{i_1}, \ldots, x_{i_l}]$. 
Otherwise, repeat the following step until obtaining either a reduced system or a system consisting entirely of monomials. 

Choose a polynomial $p = a_d x_{i_l}^d + \cdots + a_1 x_{i_l}+ a_0 \in S_1$, where each $a_i \in K[x_{i_1}, \ldots, x_{i_{l-1}}]$, and $a_d \in \sqrt{\ideal{S_1}}$.   
Then $a_d$ vanishes at every point in $V(\ideal{S_1})$.  
Write $p = a_d x_{i_l}^d + q$, where $q = a_{d-1} x_{i_l}^{d-1} + \cdots + a_1 x_{i_l} + a_0$. 
Then $q$ also vanishes at every point in $V(\ideal{S_1})$. 
Hence $V(\ideal{S_1 - \{p\} \cup \{a_d, q\}}) = V(\ideal{S_1})$. 
Let $S_2$ be a system of polynomials obtained by a valid choice of factors of $S_1 - \{p\} \cup \{a_d, q\}$. 
Then $V(\ideal{S_2}) \subseteq V(\ideal{S_1})$.  
Note that $S_2$ does not consist entirely of monomials, for if so then $(0,\ldots,0) \in V(\ideal{S_2} \subseteq V(\ideal{S_1}) \subseteq V(\ideal{S_0})$, a contradiction. 
If no polynomial in $S_2$ has leading coefficient in $\sqrt{\ideal{S_2}}$, 
then stop. 
Otherwise, set $S_1=S_2$ and repeat. 

In each step, we obtain a new system of polynomials by replacing a polynomial $p$ with two polynomials each of strictly smaller total degree than $p$, one of which is a monomial, the other with one less term than $p$. 
We then take a valid choice of factors, so each step ends with a system of irreducible polynomials. 
Since $r$, $u$, and $D$ are finite, this process must eventually terminate: if not with a system consisting entirely of monomials then because we have obtained a reduced system. 
Valid choices of factors in each step ensure that the variety remains non-empty, so the final system $S = \{f_1, \ldots, f_s\}$ obtained is consistent.  
Moreover, if $(\gamma_1, \ldots, \gamma_u) \in \overline{K}^u$ and $(\gamma_{i_1}, \ldots, \gamma_{i_t}) \in V(S)$, then $(\gamma_1, \ldots, \gamma_u) \in V(\ideal{S_0})$, so $V(\ideal{S}) \subseteq V(\ideal{S_0})$. 
Again, if $S$ consists entirely of monomials then $(0,\ldots,0) \in V(\ideal{S_0})$, contrary to assumption. 
Thus $S$ is a consistent reduced system. 
Clearly, by its construction, for each $f_j \in S$ and each indeterminate $i_k$, $\deg(f_j, x_{i_k}) \leq D$. 
\end{proof} 
}

\subsection{Proof of Theorem \ref{thm:EffectiveRado1}} 

Theorem \ref{thm:EffectiveRado1} follows from Lemmas \ref{lem:mainlemma} and \ref{lem:mainleminworstcase}, which in turn require the more technical Lemma \ref{lem:Norm_is_consistent}. 

\begin{lem} \label{lem:mainlemma} 
Let $K$ be a field of characteristic $0$, and let $f_1, \ldots, f_s$ be polynomials in the ring $K[x_1, \ldots, x_t]$ of polynomials over $K$.  
Assume that the system is consistent, and that $\deg(f_j,x_i) \leq D$ for each $i, j$.  
Then there is a solution $(\gamma_1, \ldots, \gamma_t) \in \overline{K}^t$ 
to $f_1= \ab 0, \ab \ldots, \ab f_s \ab =0$ such that 
\[ 
[K(\gamma_1, \ldots, \gamma_t) : K ] \leq 2^{2^t-t-1} D^{2^t-1} .
\]  
\end{lem}

\begin{lem} \label{lem:mainleminworstcase}
Let $K$ be a field of characteristic $p>0$, and let $f_1, \ldots, f_s$ be polynomials in the ring $K[x_1, \ldots, x_t]$ of polynomials over $K$.  
Assume that the system is consistent, and that $\deg(f_j,x_i) \leq D$ for each $i, j$.  
Then there is a solution $(\gamma_1, \ldots, \gamma_t) \in \overline{K}^t$ 
to $f_1= \ab 0, \ab \ldots, \ab f_s \ab =0$ such that 
\[ 
[K(\gamma_1, \ldots, \gamma_t) : K ] \leq  2^{3 \cdot 2^{t-1} - 2t - 1} D^{3 \cdot 2^{t-1} - 2}.
\]  
\end{lem} 

The proofs of Lemmas \ref{lem:mainlemma} and \ref{lem:mainleminworstcase} are by induction on $t$.  
Lemma \ref{lem:Norm_is_consistent} below provides the required tool for the inductive step. 
Each polynomial $f_j$ may be considered as a single-variable polynomial in $x_t$ with coefficients in the field $K(x_1, \ldots, x_{t-1})$.  
Writing $K_0 = K(x_1, \ldots, \ab x_{t-1})$ for this field, we have $f_j \in K_0[x_t]$. 
We sometimes write $f_j(x_t)$ to indicate that we are considering $f_j$ as a single-variable polynomial in $x_t$ with coefficients in $K_0$. 
Assume 
$f_s(x_t)$ 
is irreducible and separable over $K_0$.  
Let $K_1$ be the splitting field in $\overline{K_0}$ for 
$f_s(x_t)$ 
over $K_0$. 
Suppose $\deg(f_s, x_t)=d$ and $a_d$ is the leading coefficient of $f_s$.  
Let $f = (1/a_d) f_s$.  
Then $f$ splits over $K_1$, so 
\[ f = \prod_{i=1}^d (x_t-\alpha_i) \]
for some elements $\alpha_i$ in $K_1$, and the $\alpha_i$ are the roots of both 
$f(x_t)$
and 
$f_s(x_t)$ 
in $K_1$.  
It will be important for us that these roots $\alpha_i$ are distinct.  
Put $\alpha=\alpha_1$. 
Substituting $x_t=\alpha$ in each polynomial 
$f_j(x_t) \in K_0[x_t]$ 
yields a polynomial 
$f_j(\alpha)$, 
which is an element of $K_1$. 
Applying the norm to each of these elements, we obtain an element of $K_0$, 
\[
\Norm_{K_1/K_0} f_j(\alpha) = \frac{g_j(x_1, \ldots, x_{t-1})}{h_j(x_1, \ldots, x_{t-1})} \in K_0
\]
where $g_j, h_j \in K[x_1, \ldots, x_{t-1}]$.  
Place an order on monomials\ab---say, reverse lexicographic---
and insist that $g_j$ and $h_j$ share no common factor, and that $g_j$ be monic with respect to this order.  
As $K_0[x_t]$ is a unique factorisation domain, this guarantees that the expression $g_j/h_j$ is unique. 
Denote by $N(\alpha, f_j)$ the polynomial $g_j \in K[x_1, \ldots, x_{t-1}]$ obtained in this way: 

\begin{defn} \label{Dan} 
For each polynomial $f(x_t) \in K_0[x_t]$, define $N(\alpha,f)$ to be the unique polynomial $g \in K[x_1, \ldots, x_{t-1}]$ for which 
$\Norm_{K_1/K_0} f(\alpha) = {g}/{h}$, 
where $g$ and $h$ share no common factor and $g$ is monic with respect to the reverse lexicographic order on monomials. 
\end{defn}

Note that $N(\alpha, f_s)$ is the zero polynomial.  

\begin{lem} \label{lem:Norm_is_consistent}
Let $f_1, \ldots, f_s \in K[x_1, \ldots, x_t]$ be a consistent 
reduced 
system of polynomials.  
Let $K_0 = K(x_1, \ldots, x_{t-1})$, and 
assume $f_s$, considered as a polynomial in $x_t$ with coefficients in $K_0$, is separable over $K_0$. 
Let $K_1$ be the splitting field in $\overline{K_0}$ for $f_s$ over $K_0$, and 
let $\alpha \in K_1$ be a root of $f_s$.  
Then the system of polynomials $N(\alpha,f_1), \ab \ldots, \ab N(\alpha,f_{s-1}) \ab \in K[x_1, \ldots, x_{t-1}]$ is consistent. 
\end{lem} 


\begin{proof}
Let $P = \sqrt{\ideal{f_1, \ldots, f_s}}$ be the radical ideal of the ideal  generated by $f_1, \ldots, f_s$ in $K[x_1, \ldots, x_t]$.  
Write $f_s = a_d x_t^d + \cdots + a_1 x_t + a_0$, where each $a_i \in K[x_1, \ldots, x_{t-1}]$ and $a_d \neq 0$.  
Since the system is reduced, $f_s$ is irreducible and $a_d \notin P$.  
Let $f_s' = (1/a_d) f_s$. 
Polynomials $f_s$ and $f_s'$ have the same roots $\alpha_1, \ldots, \alpha_d \in K_1$.  
Put $\alpha = \alpha_1$. 

Let $P_\alpha = \{ g(x_1, \ldots, x_{t-1}, \alpha) : g \in P\}$.  
Then $P_\alpha$ is an ideal of $K[x_1, \ldots, x_{t-1}][\alpha]$. 
Let $Q = P_\alpha \cap K[x_1, \ldots, x_{t-1}]$.  
Let $S = \{a_d^k : k \in \ZZ_{\geq 0}\}$, 
and let $S\inv P_\alpha$ be the ideal 
\[ \left\{ \frac{p_\alpha}{b} : p_\alpha \in P_\alpha, b \in S \right\} \] 
in the ring 
\[ S\inv K[x_1, \ldots, \ab x_{t-1}][\alpha] = \left\{ \frac{f}{b} : f \in K[x_1, \ldots, x_{t-1}][\alpha], b \in S \right\} . \] 
If $L \supseteq K$ is a field extension, and $A \subseteq L$, denote by $\Norm_{L/K} A$ the set $\{ c \in K : c = \Norm_{L/K} a$ for some $a \in A\}$.  

\begin{claim} \label{Claire} 
For each $j \in \{1, \ldots, s-1\}$, $N(\alpha,f_j) \in Q$.
\end{claim} 

\begin{proof}[Proof of claim]
Write 
\begin{align*} 
f_s'(x_t) 
&= (x_t-\alpha_1)(x_t-\alpha_2) \cdots (x_t-\alpha_d) \\
&= x_t^d + \epsilon_{d-1}(\alpha_1, \ldots, \alpha_d)x_t^{d-1} + \epsilon_{d-2}(\alpha_1, \ldots, \alpha_d)x_t^{d-2} + \cdots + \epsilon_0(\alpha_1, \ldots, \alpha_d)
\end{align*}
where each $\epsilon_i$ is an elementary symmetric polynomial in $\alpha_1, \ldots, \alpha_d$.  
Comparing coefficients, we see that $\epsilon_i(\alpha_1, \ldots, \alpha_d) = a_{i}/a_d$.  

Let $F \in S\inv P_\alpha$.  
Then $F = {g}/{b}$ for some $g \in P_\alpha$ and $b \in S$. 
Since the norm respects multiplication (and $1/a_d^k \in K_0$ for all integers $k$), we just need consider 
$\Norm_{K_1/K_0} f(\alpha)$ 
where 
$f(\alpha) \in K[x_1, \ab \ldots, \ab x_{t-1}] [\alpha]$ 
is an irreducible factor of the numerator of $F$.  
By Lemmas \ref{lem:splitting_iff_Galois} and \ref{lem:NorminGaloisext},  
\[
\Norm_{K_1/K_0} f(\alpha) = \prod_{\sigma \in \Gal(K_1/K_0)} \sigma(f(\alpha)).
\]
Since each $\sigma \in \Gal(K_1/K_0)$ fixes $K_0$ and permutes $\alpha_1, \ldots, \alpha_d$, and $\Gal(K_1/K_0)$ acts transitively on $\alpha_1, \ldots, \alpha_d$, $\Norm_{K_1/K_0} f$ is given by 
\[
\prod_{\sigma \in \Gal(K_1/K_0)} f(x_1, \ldots, x_{t-1}, \sigma(\alpha))
\]
and this expression is symmetric in $\alpha_1, \ldots, \alpha_d$.  
Hence $\Norm_{K_1/K_0} f(\alpha)$ can be written as a polynomial $G$ in the elementary symmetric polynomials $\epsilon_i$ 
\cite[Theorem 1.12]{MR1876804}
and we have 
\begin{align*} 
\Norm_{K_1/K_0} f(\alpha)
&= G\left( \epsilon_{d-1}(\alpha_1, \ldots, \alpha_d), \ldots, \epsilon_0(\alpha_1, \ldots, \alpha_d)\right) \\
&= G\left(\frac{a_{d-1}}{a_d}, \ldots, \frac{a_0}{a_d}\right)
\end{align*}
where $G$ is a polynomial in $K[x_1, \ldots, x_{t-1}][X_1,\ldots,X_d]$. 
This shows that 
\[ \Norm_{K_1/K_0} F \in S\inv K[x_1, \ldots, x_{t-1}]. \]
Since one of the automorphisms $\sigma \in G$ is the identity, it follows that $\Norm_{K_1/K_0} F \in S\inv P_\alpha$.  
That is, 
\[\Norm_{K_1/K_0} F \in S\inv P_\alpha \cap S\inv K[x_1, \ldots, x_{t-1}]. \]
Now $f \in S\inv P_\alpha \cap S\inv K[x_1, \ldots, x_{t-1}]$ if and only if 
\[f = \frac{g(x_1, \ldots, x_{t-1})}{a_d^k}\] 
for some polynomial 
$g \in P_\alpha \cap K[x_1, \ldots, x_{t-1}] = Q$ 
and positive integer $k$. 
That is, $S\inv P_\alpha \cap S\inv K[x_1, \ldots, x_{t-1}] = S\inv Q$. 
That is, $\Norm_{K_1/K_0} F \in S\inv Q$. 
Thus $\Norm_{K_1/K_0} S\inv P_\alpha \subseteq S\inv Q$.
Since $f_j(x_1, \ldots, x_{t-1}, \alpha) \in S\inv P_\alpha$, for each $j$, 
$\Norm_{K_1/K_0} f_j(x_1, \ldots, x_{t-1}, \alpha) \in S\inv Q$. 
Hence (recall Definition \ref{Dan}) $N(\alpha,f_j) \in Q$. 
\end{proof}

\begin{claim} 
$Q$ is an ideal of $K[x_1, \ldots, x_{t-1}]$. 
\end{claim} 

\begin{proof}[Proof of claim]
Let $g, h \in Q =P_\alpha \cap K[x_1, \ldots, x_{t-1}]$ and let $r \in K[x_1, \ldots, \ab x_{t-1}]$.  
Then $g, h \in P_\alpha$, so there are polynomials $g', h' \in P$ such that $g'(x_1, \ldots, \ab x_{t-1},\alpha)=g$ and $h'(x_1, \ldots, x_{t-1},\alpha)=h$.  
Since $P$ is an ideal of $K[x_1, \ldots, x_t]$, $g'+h' \in P$.  
Also $rg' \in P$, since $r, g' \in K[x_1, \ldots, x_{t}]$.  
Then $g'+h'$ and $rg'$ when evaluated at $x_t=\alpha$ are in $P_\alpha$; that is, $g+h$ and $rg$ are in $P_\alpha$. 
Since $g, h \in K[x_1, \ldots, x_{t-1}]$, also $g+h, rg \in K[x_1, \ldots, x_{t-1}]$.  
Hence $g+h$ and $rg$ are both in $P_\alpha \cap K[x_1, \ldots, x_{t-1}] = Q$. 
\end{proof} 

Hence if $1 \notin Q$, then 1 is not in the ideal generated by the system of polynomials $N(\alpha,f_1), \ldots, N(\alpha,f_{s-1})$, and so by the weak Nullstellensatz, the system $N(\alpha,f_1), \ldots, N(\alpha,f_{s-1})$ is consistent. 
So suppose, for a contradiction, that $1 \in Q$.  
This occurs if and only if $1 \in P_\alpha$.  
Then there is a polynomial $f \in P$ with $f(x_1, \ldots, x_{t-1}, \alpha) = 1$.  
Since
\[ 
P_\alpha \subseteq K[x_1, \ldots, x_{t-1}][\alpha] \subseteq K(x_1, \ldots, x_{t-1})(\alpha) \cong K(x_1, \ldots, x_{t-1})[x_t]/\ideal{f_s} 
\] 
we have 
\[ 
f(x_1, \ldots, x_{t-1}, x_t) - 1 \in \ideal{f_s} \subseteq K(x_1, \ldots, x_{t-1})[x_t].
\] 
Hence there is a polynomial 
$g \in K(x_1, \ldots, x_{t-1})[x_t]$ 
such that 
$f - 1 = g f_s$. 
Each coefficient of $g$ is a rational expression in indeterminants $x_1, \ldots, x_{t-1}$; write $g = n/m$ where $m$ is the least common multiple of the denominators of the coefficients of $g$. 
We may assume $n$ and $m$ have no common factor. 
Since $gf_s \in K[x_1, \ldots, x_t]$, 
$m$ must be factor of $f_s$.  
But $f_s$ is irreducible, so $m$ is a unit. 
That is, $g \in K[x_1, \ldots, x_t]$. 
Choose a point $\gamma \in V(P)$. 
Now 
\[
f(\gamma) - 1 = g(\gamma) f_s(\gamma) 
\]
implies $-1 = 0$, a contradiction. 
\end{proof} 

\begin{proof}[Proof of Lemma \ref{lem:mainlemma}]
We proceed by induction on $t$.  
The result clearly holds for $t=1$.  
As in the proof of Lemma \ref{lem:Norm_is_consistent}, let $P = \sqrt{\ideal{f_1, \ldots, f_s}}$ be the radical ideal of the ideal  generated by $f_1, \ldots, f_s$ in $K[x_1, \ldots, x_t]$, 
and let $K_0 = K(x_1, \ldots, x_{t-1})$. 
Applying Lemma \ref{lem:reduced_system}, we may assume that $f_s$ is irreducible, has leading indeterminate $x_t$, and has leading coefficient $a_d \in K[x_1, \ldots, x_{t-1}]$ with $a_d \notin P$.  
As in the proof of Lemma \ref{lem:Norm_is_consistent}, 
write $f_s = a_d x_t^d + \cdots + a_0$ 
and consider $f_s$ as a polynomial in $K_0[x_t]$; 
let $K_1$ be the splitting field in $\overline{K_0}$ for $f_s$ over $K_0$, and let $\alpha \in K_1$ be a root of $f_s$. 
Again as in the proof of Lemma \ref{lem:Norm_is_consistent}, 
let $f_s' = (1/a_d) f_s$, 
let $P_\alpha = \{ g(x_1, \ldots, x_{t-1}, \alpha) : g \in P\}$ and let $Q = P_\alpha \cap K[x_1, \ldots, x_{t-1}]$.  
As in the proof of the first claim in the proof of Lemma \ref{lem:Norm_is_consistent}, we have, for each $j \in \{1, \ldots, s\}$, 
\begin{align*} 
\Norm_{K_1/K_0} f_j(\alpha)
&=\prod_{\sigma \in \Gal(K_1/K_0)} f_j(x_1, \ldots, x_{t-1}, \sigma(\alpha)) \\
&= G_j \left( \epsilon_{d-1}(\alpha_1, \ldots, \alpha_{d}), \ldots, \epsilon_0(\alpha_1, \ldots, \alpha_{d})\right) \\
&= G_j\left(\frac{a_{d-1}}{a_d}, \ldots, \frac{a_0}{a_d}\right)
\end{align*}
where $G_j$ is a polynomial in $K[x_1, \ldots, x_{t-1}][X_1,\ldots,X_d]$. 
Since 
the degree in $\Norm_{K_1/K_0} f_j(\alpha)$ of each root $\alpha_k$ is at most $D$, and 
the degree of each $\alpha_k$ in the symmetric polynomials 
is 1, the degree of each $X_i$ in $G_j(X_1, \ldots, X_k)$ is at most $D$.  
Since the degree of each indeterminate in each coefficient of $G_j$ is at most $D^2$, and the degree of each $x_i$ in each coefficient $a_i$ of $f_j$ is at most $D$, the degree of each indeterminate in the numerator of $\Norm_{K_1/K_0} f_j(\alpha)$ is at most $2D^2$. 
Thus the system 
\[
N(\alpha,f_1), \ldots, N(\alpha,f_{s-1}) \in K[x_1, \ldots, x_{t-1}] 
\]
has no indeterminate $x_i$ of degree more than $2D^2$.  
By Lemma \ref{lem:Norm_is_consistent}, it 
is consistent.  
By induction, this system 
has a solution $(\gamma_1, \ldots, \gamma_{t-1}) \in \overline{K}^{t-1}$ with $[ K(\gamma_1, \ldots, \gamma_{t-1}) : K ]$ at most 
\[
2^{2^{t-1}-(t-1)-1} (2D^2)^{2^{t-1}-1} .
\]

Observe that for each $j$ there is a positive integer $m_j$ such that $N(\alpha,f_j) = (a_d)^{m_j} \Norm_{K_1/K_0} f_j(\alpha)$. 
For each $\sigma_i \in \Gal_{K_1/K_0}$, $i \in \{1, \ldots, d\}$, write $\alpha_i = \sigma_i(\alpha)$ with $\alpha = \alpha_1$. 
Consider the product 
\[
(a_d)^{m_j} \Norm_{K_1/K_0} f_j(x_1, \ldots, x_{t-1}, \alpha_1) = (a_d)^{m_j} \prod_{i=1}^d f_j(x_1, \ldots, x_{t-1}, \alpha_i). 
\]
Evaluating at $x_1=\gamma_1$, \ldots, $x_{t-1}=\gamma_{t-1}$ (working in $\overline{K_0}$), 
we obtain $0$, because this product is equal to $(a_d)^{m_j} N(\alpha_1,f_j)$ and evaluating $N(\alpha_1, f_j)$ at $(\gamma_1, \ldots, \gamma_{t-1})$ 
yields $0$. 

\begin{claimnon} \label{handyman} 
$a_d$ does not evaluate to zero at $(\gamma_1, \ldots, \gamma_{t-1})$. 
\end{claimnon}

\begin{proof}[Proof of claim]
Let $P_\alpha$ and $Q$ be as in the proof of Lemma \ref{lem:Norm_is_consistent}. 
Suppose $a_d$ evaluates to zero at $(\gamma_1, \ldots, \gamma_{t-1})$. 
Then there is a positive integer $m$ such that $(a_d)^m \in \ideal{N(\alpha, f_1), \ab \ldots, N(\alpha, f_{s-1})}$. 
Since $\ideal{N(\alpha, f_1), \ldots, N(\alpha, f_{s-1})} \subseteq Q$, this implies $(a_d)^m \in Q$. 
But $(a_d)^m \in Q$ if and only if $(a_d)^m \in P_\alpha$, which occurs if and only if $(a_d)^m \in P$, and so if and only if $a_d$ is in $P$. 
But $a_d$ is not in $P$, so this is a contradiction. 
\end{proof}

Since $a_d$ does not evaluate to zero at $(\gamma_1, \ldots, \gamma_{t-1})$, there is an $i \in \{1, \ldots, d\}$ for which the factor $f_j(\gamma_1, \ldots, \gamma_{t-1}, \alpha_i)$ is zero. 
Since 
\[ 
K(x_1, \ldots, x_{t-1})[\alpha_i] \ab \cong K(x_1, \ldots, x_{t-1})[x_t] / \ideal{f_s} 
\] 
this occurs if and only if there is a polynomial $g_j \in K(x_1, \ldots, x_{t-1})[x_t]$ such that 
\[ 
f_j(\gamma_1, \ldots, \gamma_{t-1}, x_t) = g_j(\gamma_1, \ldots, \gamma_{t-1}, x_t) \cdot f_s(\gamma_1, \ldots, \gamma_{t-1}, x_t).
\]

Since $f_s(\gamma_1, \ldots, \gamma_{t-1}, x_t)$ has degree at most $D$ in $x_t$, it has a root $\gamma_t \in \overline{K}$ with $[K(\gamma_1, \ldots, \gamma_t):K(\gamma_1, \ldots, \gamma_{t-1})] \leq D$. 
Thus $(\gamma_1, \ldots, \gamma_t) \in \overline{K}^t$ is a solution to our original system $f_1, \ldots, f_s$, 
and 
\begin{align*} 
[K(\gamma_1, \ldots, \gamma_t) : K] 
&= [K(\gamma_1, \ldots, \gamma_t) : K(\gamma_1, \ldots, \gamma_{t-1})] [K(\gamma_1, \ldots, \gamma_{t-1}) : K] \\ 
&\leq D \cdot 2^{2^{t-1}-(t-1)-1} (2D^2)^{2^{t-1}-1} \\
&= 2^{2^t-t-1} D^{2^t-1}  \qedhere
\end{align*}
\end{proof} 

We now apply the same induction argument in the case that the field $K$ has positive characteristic $p$.  
We just require an additional step in order to deal with the possibility that the polynomials in our system are all inseparable over $K(\mathbf{x}-i)[x_i]$, for every $i$.  
By Lemma \ref{lem:if_nonseparable}, if this is the case, then the exponent on every indeterminate in every term of every polynomial in the system is a multiple of $p$. 

\begin{proof}[Proof of Lemma \ref{lem:mainleminworstcase}]
We proceed by induction on $t$.  
The result clearly holds for $t=1$.  
Applying Lemma \ref{lem:reduced_system}, we may assume that the system is reduced. 

Let $q$ be the largest multiple of $p$ that is a common factor of all exponents of $x_t$ among all terms of $f_1, \ldots, f_s$, so that for each $j$, $f_j = g_j(x_t^q)$, where $g_j \in K[x_1, \ldots, x_{t-1}][x_t]$ is irreducible. 
Let $z = x_t^q$, and consider the system of polynomials $g_1, \ldots, g_s \in K[x_1, \ldots, x_{t-1}, z]$ obtained by replacing each polynomial $f_j$ with $g(z)$.  
We may assume (renaming polynomials if necessary) that $g_s$ has at least one term in which the exponent on $z$ not a multiple of $p$. 
We now have a system $g_1, \ldots, g_s \in K[x_1, \ldots, x_{t-1}, z]$, in which (by Lemma \ref{lem:if_nonseparable}) $g_s$ is separable over $K(x_1, \ldots, x_{t-1})$. 

Write $g_s = a_d z^d + \cdots + a_0$.  
Since each $g_j$ is obtained from $f_j$ by just replacing $x_t^q$ with $z$, and $a_d \notin \sqrt{\ideal{f_1, \ldots, f_s}}$, it is also the case that $a_d \notin \sqrt{\ideal{g_1, \ldots, g_s}}$. 
Let $P = \sqrt{\ideal{g_1, \ldots, g_s}}$, let $K_0 = K(x_1, \ldots, x_{t-1})$, let  $K_1$ be the splitting field in $\overline{K_0}$ for $g_s$ over $K_0$, and let $\alpha \in K_1$ be a root of $g_s$, as in Lemma \ref{lem:Norm_is_consistent}.  
Again as in the proof of the first claim in the proof of Lemma \ref{lem:Norm_is_consistent}, we have 
\begin{align*} 
\Norm_{K_1/K_0} g_j(\alpha) 
&=\prod_{\sigma \in \Gal(K_1/K_0)} g_j(x_1, \ldots, x_{t-1}, \sigma(\alpha)) \\
&= G_j \left( \epsilon_{d-1}(\alpha_1, \ldots, \alpha_{d}), \ldots, \epsilon_0(\alpha_1, \ldots, \alpha_{d})\right) \\
&= G_j\left(\frac{a_{d-1}}{a_d}, \ldots, \frac{a_0}{a_d}\right)
\end{align*}
for some polynomial  $G_j \in K[x_1, \ldots, x_{t-1}][X_1,\ldots,X_d]$. 
Just as in the proof of Lemma \ref{lem:mainlemma}, 
the system 
\[
N(\alpha,g_1), \ldots, N(\alpha,g_{s-1}) \in K[x_1, \ldots, x_{t-1}] 
\]
is consistent by Lemma \ref{lem:Norm_is_consistent}, and has no indeterminate $x_i$ of degree more than $2D^2$.  
By induction, this system 
has a solution $(\gamma_1, \ldots, \gamma_{t-1})$ with 
\[ 
[K(\gamma_1, \ldots, \gamma_{t-1}) : K ] 
\leq  2^{3 \cdot 2^{t-2} - 2(t-1) - 1} (2D^2)^{3 \cdot 2^{t-2} - 2} .
\]  
Hence 
by the argument in the proof of Lemma \ref{lem:mainlemma}, 
the system $g_1, \ldots, g_s \in K[x_1, \ldots, x_{t-1}, z]$ has a solution $(\gamma_1, \ldots, \ab \gamma_{t-1}, \ab \gamma_z)$ with 
\[
[K(\gamma_1, \ldots, \gamma_{t-1}, \gamma_z) : K] \leq D \cdot [K(\gamma_1, \ldots, \gamma_{t-1}) : K ] .
\]
Now $(\gamma_1, \ldots, \gamma_{t-1}, \sqrt[q]{\gamma_z})$ is a solution to our original system.  
The minimal polynomial of $\sqrt[q]{\gamma_z}$ over $K(\gamma_1, \ldots, \gamma_{t-1}, \gamma_z)$ divides $X^q-\gamma_z$, and $q \leq D$, so 
\begin{multline*} 
[ K(\gamma_1, \ldots, \gamma_{t-1}, \gamma_z, \sqrt[q]{\gamma_z}) : K ] = \\
[K(\gamma_1, \ldots, \gamma_{t-1}, \gamma_z, \sqrt[q]{\gamma_z}) : K(\gamma_1, \ldots, \gamma_{t-1}, \gamma_z) ] \cdot [K(\gamma_1, \ldots, \gamma_{t-1}, \gamma_z) : K] \\ 
\leq D \cdot [K(\gamma_1, \ldots, \gamma_{t-1}) : K ] \cdot D 
\leq 2^{3 \cdot 2^{t-2} - 2(t-1) - 1} (2D^2)^{3 \cdot 2^{t-2} - 2} \cdot D^2 \\
= 2^{3 \cdot 2^{t-1} - 2t - 1} D^{3 \cdot 2^{t-1} - 2} .
\end{multline*} 
Hence, taking $\gamma_t = \sqrt[q]{\gamma_z}$, certainly also 
\[ 
[ K(\gamma_1, \ldots, \gamma_{t-1}, \gamma_t) : K ] \leq 
2^{ 3 \cdot 2^{t-1} - 2t - 1} D^{3 \cdot 2^{t-1} - 2 } .
\qedhere 
\]
\end{proof} 

\begin{proof}[Proof of Theorem \ref{thm:EffectiveRado1}]  
Together, Lemmas \ref{lem:mainlemma} and \ref{lem:mainleminworstcase} guarantee that given an arbitrary system of polynomials over a field $K$, in $t$ variables, with each variable of degree at most $D$, 
there is always an algebraic extension of $K$ of degree at most 
\begin{equation}  \label{eqn:degree_bound}
2^{ 3 \cdot 2^{t-1} - 2t - 1} D^{3 \cdot 2^{t-1} - 2 }
\end{equation}
in which we can find a solution to the system.  

Given a rank-$r$ matroid on $n$ elements, an associated system of polynomials has, 
in each polynomial coming from a determinant, every variable of degree at most 1,  
and at most ${n \choose r}$ determinantal polynomials.  
Hence we have $t \leq nr + 1 \leq n^2 + 1$ and $\deg(f_i, x_j) \leq {n \choose r} \leq 2^n$ for each $i$, $j$.  
Hence the bound given in (\ref{eqn:degree_bound}) yields (for $n \geq 2$) 
\begin{align*} 
2^{ 3 \cdot 2^{t-1} - 2t - 1} D^{3 \cdot 2^{t-1} - 2 } 
&\leq 
2^{ 3 \cdot 2^{n^2} - 2(n^2+1) - 1} (2^n)^{3 \cdot 2^{n^2} - 2 } \\ 
&= 2^{ 3n2^{n^2}+3\cdot2^{n^2}-2n^2-2n-3 } \\
&< 2^{3n2^{n^2+1}} < 2^{2^{2n^2}}.
\qedhere
\end{align*} 
\end{proof}

\section{Positive characteristic} 
\label{sec:not_rep_over_char_zero}

Let $c_{>0}(n) = \max\{c(M) : M$ is representable only over a field of positive characteristic\ab$\}$ and let 
$f_{>0}(n) = \max\{f(M) : M$ is representable only over a field of positive characteristic$\}$. 
We obtain the following bounds.  

\begin{thm} \label{thm:bounds_not_char_zero} 
For all positive integers $n$, 
\[
\log_2 \log_2 c_{>0}(n) < n^4 \ \text{ and } \ \log_2 \log_2 \log_2 f_{>0}(n) < n^3.
\]
\end{thm}

Theorem \ref{thm:bounds_not_char_zero} just combines the statements of Theorems \ref{thm:upperboundoncharnotrepoverC} and \ref{thm:upperboundonfnotrepoverC} below.  
Let $M$ be a representable matroid, but not over characteristic $0$.  
Applying a result of Krick, Pardo, and Sombra \cite{MR1853355} gives the following bound on $c(M)$.  

\begin{thm} \label{thm:upperboundoncharnotrepoverC}
Let $M$ be an $n$-element matroid representable only over strictly positive characteristic.  
Then 
\[
\log_2 \log_2 c(M) < n^4 .
\]
\end{thm}

We obtain this bound as follows. 
Let $F \subseteq \ZZ[x_1, \ldots, x_t]$ be the system of polynomials given by $M$ as described at the beginning of Section \ref{sec:bounding_degree_field_extension}.    
Denote by $\ideal{F}$ the ideal in $\ZZ[x_1, \ldots, x_t]$ generated by the polynomials in $F$.  
Let $K$ be a field, and denote by $F_K$ the system of polynomials $F$ viewed over the polynomial ring $K[x_1, \ldots, x_t]$, and by $\ideal{F_K}$ the ideal generated by $F_K$ in $K[x_1, \ldots, x_t]$.  
Hilbert's weak Nullstellensatz says that $F_K$ is solvable over some extension field of $K$ if and only if $1 \notin \ideal{F_K}$.  
If $1 \in \ideal{F}$, then also $1 \in \ideal{F_K}$ for all fields $K$, so $M$ is not representable over any field.  
But suppose $\ideal{F}$ contains an integer $a > 1$.  
Then the system $F_K$ is solvable in $K$ only if the characteristic of $K$ divides $a$.  
In other words, if $M$ can be represented over $K$, then the characteristic of $K$ divides $a$.  
Thus $a$ provides an upper bound on $c(M)$.  

One way to state Hilbert's Nullstellensatz is the following.  

\begin{thmnon}[Hilbert's Nullstellensatz] 
Let $f_1, \ldots, f_s \in \ZZ[x_1, \ldots, x_t]$ be polynomials such that the system 
$f_1 = 0, \ldots, f_s = 0$ 
has no solution in $\CC^t$.  
Then there is a positive integer $a \in \ideal{f_1, \ldots, f_s}$.  
\end{thmnon}

The result of Krick, Pardo, and Sombra we use 
is the following effective version of Hilbert's Nullstellensatz.  
For a polynomial $f \in \ZZ[x_1, \ldots, x_t]$, let $\deg f$ denote its total degree, and let $h(f) = \log H(f)$ denote the logarithm of the maximum absolute value of its coefficients.  

\begin{thm}[\cite{MR1853355}] \label{thm:SharpestimatesNullstellensatz} 
Let $f_1, \ldots, f_s \in \ZZ[x_1, \ldots, x_t]$ be polynomials such that the system 
$f_1 = 0, \ldots, f_s = 0$ 
has no solution in $\CC^t$.  
Set $d = \max_i \deg f_i$ and $h = \max_i h(f_i)$.  
Then there is a positive integer $a \in \ideal{f_1, \ldots, f_s}$ satisfying 
\[ \log a \leq 4t(t+1)d^t \left(h + \log s + (t + 7)\log(t+1) \, d\right) .\] 
\end{thm}

\begin{proof}[Proof of Theorem \ref{thm:upperboundoncharnotrepoverC}] 
By Lemma \ref{lem:boundsonsystemparameters}, 
for our system $F \subseteq \ZZ[x_1, \ldots, x_t]$ we have 
$s \leq 2^n$, 
$d \leq n 2^n$, 
$t \leq n^2 + 1$, and 
$H \leq n^{n2^n}$. 
Hence 
\[
h \leq \log n^{n2^n} < n2^n \log_2 n \leq n2^n\log_2 2^n = n^2 2^n \leq 2^n 2^n = 2^{2n}.
\]
Substituting these values into the result of Theorem \ref{thm:SharpestimatesNullstellensatz} 
we obtain a positive integer $a \in \ideal{f_1, \ldots, f_s}$ satisfying  
\begin{align*} 
\log a &\leq 4(n^2+1)(n^2+2)(n2^n)^{n^2+1} \left(2^{2n} + \log 2^n + (n^2+8)\log(n^2+2) \, n2^n \right) \\
&\leq (4n^4+12n^2+8) (n^{n^2+1}2^{n^3+n}) (2^nn(n^2+8)\log(n^2+2)+n\log2+2^{2n}) \\
&\leq (4n^4+12n^2+8) (n^{n^2+1}2^{n^3+n}) \left(2^n\left(n(n^2+8)\log(n^2+2)+n\right)+2^{2n}\right) .
\end{align*}
Using the facts 
$n^{n^2+1} \leq 2^{n^3}$, 
$n(n^2+8)\log(n^2+2)+n \leq n^4$, 
$(4n^4+12n^2+8)(n^4+1) \leq n^9$, and 
$n^9 \leq 2^{4n}$, 
we obtain 
\begin{align*} 
\log a 
&\leq (4n^4+12n^2+8) (2^{2n^3+n}) (2^{2n}(n^4+1)) \\ 
&\leq n^9 2^{2n^3+3n} 
\leq 2^{4n} 2^{2n^3+3n} = 2^{2n^3+7n} . 
\end{align*}
Hence 
\[
\log_2 a < 2 \cdot \log a 
< 2 \cdot 2^{2n^3+7n} = 2^{2n^3+7n+1} \leq 2^{n^4} .  
\qedhere
\]
\end{proof}

\begin{thm} \label{thm:upperboundonfnotrepoverC}
Let $M$ be an $n$-element matroid representable only over strictly positive characteristic.  
Then 
\[
\log_2 \log_2 \log_2 f(M) < n^3 .
\]
\end{thm}

\begin{proof} 
By Theorem \ref{thm:upperboundoncharnotrepoverC}, $M$ is representable over a field of characteristic $p$, where $p$ is a prime of size at most 
$2^{2^{n^4}}$. 
Hence by Theorem \ref{thm:EffectiveRado1}, $M$ is representable over a simple algebraic extension of $\GF(p)$ of degree at most 
$N = 2^{2^{2n^2}}$. 
That is, $M$ is representable over a field of size at most $p^N$.  
So 
\[
f(M) 
\leq ( 2^{2^{n^4}} )^{2^{2^{2n^2}}} 
= 2^{2^{n^4+2^{2n^2}}} \leq 2^{2^{2^{n^3}}} .
\qedhere
\]
\end{proof}

\section{Characteristic zero} 
\label{sec:rep_over_char_zero}

Let $c_0(n) = \max\{c(M) : M$ is representable over a field of characteristic $0\}$ and let 
$f_0(n) = \max\{f(M) : M$ is representable over a field of characteristic $0\}$. 
We obtain the following bounds.  

\begin{thm} \label{thm:bounds_char_zero} 
For all positive integers $n$, 
\[
\log_2 \log_2 c_0(n) < n^5 \ \text{ and } \ \log_2 \log_2 \log_2 f_0(n) < n^3 . 
\]
\end{thm}

We use the following two results.  
The first combines and paraphrases a result of Koll\'ar \cite{MR944576} and a result of Sombra \cite{MR1659402} giving bounds on the degree of polynomials in B\'ezout's identity. 

\begin{thm}[\cite{MR944576, MR1659402}] \label{thm:KollarSombra}
Let $K$ be a field, and let $f_1, \ldots, f_s \in K[x_1, \ldots, x_t]$ be polynomials each of total degree at least 1 and at most $d$. 
Suppose $f_1, \ldots, f_s$ have no common zero in $\overline{K}^t$. 
Then there exist polynomials $g_1, \ldots, g_s \in K[x_1, \ldots, x_t]$ satisfying 
\[ g_1 f_1 + \cdots + g_s f_s = 1 \]
where each $g_i$ has total degree at most $d^t$. 
\end{thm} 

The second gives a lower bound on the product of the primes that are at most a given integer. 

\begin{thm} \label{thm:boundonproductofprimesatmostn}
Let $a$ be a positive integer. 
The product of the primes at most $a$ is greater than $2^{a-3}$. 
\end{thm} 

\begin{proof} 
By \cite[Theorem 10]{MR0137689}, $\prod_{p \leq a} p > e^{0.84a}$ 
for $a \geq 101$. 
Since $e^{0.84} > 2$, $\prod_{p \leq a} p > 2^{a}$ for $x \geq 101$. 
It is straightforward to check by direct calculation that the inequality $\prod_{p \leq a} p > 2^{a-3}$ holds for $a \leq 100$. 
\end{proof} 

We also use Hadamard's inequality, a well-known bound on the determinant of a matrix: 

\begin{lemnon}[Hadamard's inequality]
Let $A$ be an $n \times n$ matrix with entries in $\CC$. 
If every entry $A_{ij}$ of $A$ satisfies $|A_{ij}| \leq B$, then $|\det(A)| \leq B^n n^{n/2}$. 
\end{lemnon}

The \emph{height} $H(f)$ of a polynomial $f$ is the maximum of the absolute values of its coefficients. 
Theorem \ref{thm:bounds_char_zero} is a corollary of the following theorem. 

\begin{thm} \label{thm:char0boundonp}
Let $f_1, \ldots, f_s \in \ZZ[x_1, \ldots, x_t]$ be polynomials of total degree at least 1 and at most $d$, and of height at most $H$, and assume $f_1, \ldots, f_s$ share a common zero in $\CC^t$.  
Let $L = s {d^{t}+t \choose t}$. 
Then there is a prime $p$ satisfying 
\[ 
p < 6 + 2 L \log_2 H + L \log_2 L
\] 
such that $\ZZ[x_1,\ldots ,x_t] / \ideal{p,f_1,\ldots ,f_s}$ is nonzero.  
Moreover, for all $p>H^L \sqrt{L}^L$ the ring $\ZZ[x_1,\ldots ,x_t] / \ab \ideal{p,f_1,\ldots ,f_s}$  is nonzero.
\end{thm} 

\begin{proof}
Note that for a commutative ring $R$, the collection of polynomials of degree at most $d^{t}$ in $R[x_1, \ldots, x_t]$ is a free $R$-module on the generators
\[ 
S:=\{x_1^{i_1}\cdots x_t^{i_t} \colon i_1 + \cdots + i_t \le d^{t} \}. 
\] 
The size of $S$ is the number of ways to write $d^t$ as a sequence of $t+1$ non-negative integers 
(there is a 1-1 correspondence between the sequences of length $t$ whose sum is at most $d^t$ and sequences of length $t+1$ whose sum is exactly $d^t$, obtained by truncating each of the latter sequences at $t$ terms). 
So $|S|$ is the number of weak compositions of $d^t$ into $t+1$ parts; that is, $|S| = {d^{t} + t \choose t}$.

Now let $S = \{m_1, m_2, \ldots, m_{|S|}\}$. 
Let $\{ z_{i,j} : 1 \leq i \leq |S|, 1 \leq j \leq s \}$ be a set of indeterminates; 
this collection has size $L$. 
Define 
\[
g_j = \sum_{i=1}^{|S|} z_{i,j} m_i \in \ZZ[ x_1, \ldots ,x_t][z_{i,j} \colon 1 \leq i \le |S|, 1 \le j \le s ].
\]

Now consider the equation 
\begin{equation} \label{eqn:witness}
1 - g_1 f_1 + \cdots + g_s f_s = 0.
\end{equation}
By Theorem \ref{thm:KollarSombra} there is an assignment of values from a field $K$ to the indeterminates $z_{i,j}$ satisfying (\ref{eqn:witness}) if and only if $1 \in \ideal{f_1, \ldots, f_s}_K$.  
Let $t \colon \ZZ[x_1,\ldots,x_t][z_{i,j}] \to K[x_1,\ldots,x_t]$ be an assignment of values in $K$ to the indeterminates $z_{i,j}$. 
Expand (\ref{eqn:witness}) and set $t(z_{i,j})=t_{i,j} \in K$. 
Consider the coefficient of a monomial $m \in S$ appearing in this equation. 
Each such coefficient yields an equation of the form
\[ 
\delta_{m,1} - \sum_{i=1}^{|S|} \sum_{j=1}^s t_{i,j} c_{i,m,j} = 0 
\] 
where $c_{i,m,j}$ is a coefficient of $f_j$, and hence is at most $H$ in absolute value (and where $\delta_{m,1} = 1$ if $m=1$ and is otherwise $0$). 

Now write equation (\ref{eqn:witness}) as a matrix equation $A\vec{z} = \vec{b}$, where $A$ is a $|S| \times s|S|$ integer matrix (with rows indexed by the monomials in $S$ and columns by the $s|S|=L$ variables $z_{ij}$ that are the components of $\vec{z}$). 
The entries of $A$ are at most $H$ in absolute value and $\vec{b}$ has one entry equal to one and the rest equal to zero.  
Observe that, for a field $K$, $A\vec{z} = \vec{b}$ has a solution in $\overline{K}^t$ if and only if $1 \in \ideal{f_1, \ldots, f_s}_K$. 
Since $1$ is not in the ideal $\ideal{f_1,\ldots ,f_n}_\QQ$, we see that this equation $A\vec{z} = \vec{b}$ has no solutions in $\CC^t$. 
Let $r$ denote the rank of $A$.  
Then there is an $(r+1) \times (r+1)$ minor of the matrix $(A|\vec{b})$ that does not vanish.  
Since $r\le L-1$ and the entries of $(A|\vec{b})$ are at most $H$, by Hadamard's inequality this minor is bounded by $(H\sqrt{L})^L$.  
Let $D$ denote this minor. 
Then $|D| \le H^L \sqrt{L}^L$. 

On the other hand, if $p$ is prime and $1 \in \ideal{f_1,\ldots ,f_s}_{\text{GF}(p)}$ (taking reductions of the $f_i$ modulo $p$) then $A\vec{z} = \vec{b}$ has a solution modulo $p$.  
Since $A$ has rank at most $r \mod p$, then $(A|\vec{b})$ must have rank at most $r$ mod $p$ and so $D$ must vanish modulo $p$. 

In particular, this means that if $p>H^L \sqrt{L}^L$ then, as $D$ does not vanish modulo $p$, $A\vec{z} = \vec{b}$ does not have a solution modulo $p$. 
Thus $1 \notin \ideal{f_1, \ldots, f_s}_{\text{GF}(p)}$. 
In other words, $f_1, \ldots, f_s$ share a common zero in $\overline{\text{GF}(p)}^t$. 

Let $p'$ be the least prime for which $A\vec{z} = \vec{b}$ does not have a solution modulo $p'$; equivalently, let $p'$ be the least prime for which $1 \notin \ideal{f_1, \ldots, f_s}_{\text{GF}(p')}$. 
Let $q$ be the largest prime less than $p'$. 
Then $D$ is a multiple of all primes $\leq q$. 
Hence, by Theorem \ref{thm:boundonproductofprimesatmostn} and Hadamard's Inequality, 
\[ 
2^{q-3} \le \prod_{p\le q} p \leq |D| \le H^L \sqrt{L}^L
\]
which implies 
\[
q \le 3 + L \log_2 H + L/2 \log_2 L .
\]
Hence by Bertrand's postulate, $p' < 2q \leq 6 + 2L \log_2 H + L \log_2 L$. 
\end{proof} 

Now suppose our system of polynomials $f_1, \ldots, f_s \in \ZZ[x_1, \ldots, x_t]$ of Theorem \ref{thm:char0boundonp} is a system arising from an $n$-element matroid $M$, of rank $r$, representable over a field of characteristic zero, as described in Section \ref{sec:polynomials_from_a_matroid}. 
By Theorem \ref{thm:char0boundonp} there is a prime $p < 6 + 2 \log_2 H + L \log_2 L$ such that $1 \notin \ideal{p, f_1, \ldots, f_s}$. 
Since the polynomials $f_1, \ldots, f_s$, reduced modulo $p$ share a common zero in $\overline{\text{GF}(p)}^t$, $M$ is representable over a field of characteristic $p$. 
Hence 
\[ 
c(M) \leq 6 + 2L \log_2 H + L \log_2 L.
\]
To complete the proof of Theorem \ref{thm:bounds_char_zero}, we just need to write $L$ and $H$ in terms of $n$. 
By Lemma \ref{lem:boundsonsystemparameters}, for our system of polynomials $f_1, \ldots, f_s$, we have $s \leq 2^n$, $t \leq n^2+1$, $d \leq n 2^n$, and $H \leq n^{n2^n}$. 
Hence 
\begin{align*} 
L &= s{{d^t+t} \choose t} \leq s 2^{d^t+t} \leq 2^n 2^{(n2^n)^{n^2+1} + n^2 + 1} \\ 
&\leq 2^n 2^{(n^{n+1})^{n^2+1} + n^2 + 1} 
\leq 2^{n^{n^4} + n^2 + n + 1} . 
\end{align*}
Observe that $H \leq n^{n2^n} \leq 2^{2^{2n}}$, which is a more convenient bound. 

\begin{proof}[Proof of Theorem \ref{thm:bounds_char_zero}]
Let $M$ be an $n$-element matroid representable over a field of characteristic zero. 
By Theorem \ref{thm:char0boundonp}, and the above bounds for $L$ and $H$ 
\begin{align*}
c(M) 
&\leq 6 + 2L \log_2 H + L \log_2 L \\ 
&\leq 6 + 2 \cdot 2^{{n^{n^4} + n^2 + n + 1}} \log_2 2^{2^{2n}} + 2^{{n^{n^4} + n^2 + n + 1}} \log_2 2^{{n^{n^4} + n^2 + n + 1}} \\
&\leq 6 + 2^{{n^{n^4} + n^2 + n + 2}} 2^{2n} + 2^{{n^{n^4} + n^2 + n + 1}} \cdot ({n^{n^4} + n^2 + n + 1}) \\ 
&\leq 6 + 2^{{n^{n^4} + n^2 + 3n + 2}} + 2^{{n^{n^4} + n^2 + n + 1}} \cdot ({n^{n^4} + n^2 + n + 1}) \\ 
&\leq 2 \cdot 2^{{n^{n^4} + n^2 + 3n + 2}} \cdot ({n^{n^4} + n^2 + n + 1}) \\ 
&\leq 2^{{n^{n^4} + n^2 + 3n + 3}} \cdot ({n^{n^4} + n^2 + n + 1}) \\ 
&\leq 2^{{n^{n^4} + n^2 + 3n + 3}} \cdot 2^{n^5} 
= 2^{{n^{n^4} + n^5 + n^2 + 3n + 3}} \leq 2^{2^{n^5}} .
\end{align*}

Hence by Theorem \ref{thm:EffectiveRado1}
\[
f(M) \leq (2^{2^{n^5}})^{2^{2^{2n^2}}} = 2^{2^{n^5+2^{2n^2}}} \leq 2^{2^{2^{n^3}}}. \qedhere
\]
\end{proof} 

\begin{proof}[Proof of Theorem \ref{thm:boundforGFp}]
If $p>H^L L^{L/2}$, then by Theorem \ref{thm:char0boundonp} $M$ is representable over $\GF(p)$. 
Substituting $2^{2^{2n}}$ for $H$ and $2^{n^{n^4}+n^2+n+1}$ for $L$ yields 
\begin{align*} 
H^L L^{L/2} 
&\leq (2^{2^{2n}})^{2^{n^{n^4}+n^2+n+1}} \cdot (2^{n^{n^4}+n^2+n+1})^{2^{-1}2^{n^{n^4}+n^2+n+1}} \\ 
&\leq 2^{2^{n^{n^4}+n^2+3n+1}} \cdot 2^{(n^{n^4}+n^2+n+1) \cdot 2^{n^{n^4}+n^2+n}} \\ 
&\leq 2^{2^{n^{n^4}+n^2+3n+1}} \cdot 2^{(2^{n^5}) \cdot 2^{n^{n^4}+n^2+n}} \\ 
&\leq 2^{2^{n^{n^4}+n^2+3n+1}} \cdot 2^{2^{n^{n^4}+n^5+n^2+n}} \\ 
&\leq 2^{2 \cdot 2^{n^{n^4}+n^5+n^2+n}} 
= 2^{2^{n^{n^4}+n^5+n^2+n+1}} 
\leq 2^{2^{2^{n^5}}} .
\qedhere
\end{align*}
\end{proof}

\section{A lower bound} 
\label{sec:lower_bound} 

Using a result from \cite{MR664703}, we obtain the following lower bound on $c(n)$.  

\begin{thm} \label{thm:lowerboundonc} 
$\displaystyle \log_2 c(n) \geq {(n-7)/2}$
\end{thm}

The result we use is the following.  

\begin{thm}[Brylawski \cite{MR664703}, Corollary 3.3] \label{thm:Brylawski} 
For any prime $p$ there is a matroid $M$ on at most $2\floor{\log_2 p} + 6$ elements with $c(M) = p$.  
\end{thm}

\begin{proof}[Proof of Theorem \ref{thm:lowerboundonc}]
For each positive integer $n \geq 7$, choose a prime $p$ such that 
\[ 2^{(n-7)/2} \leq p \leq 2^{(n-5)/2}. \] 
By Bertrand's postulate, this is always possible.  
Since $\frac{n-5}{2}$ is $\frac{1}{2}$-integral, $\floor{\log_2 p} + \frac{1}{2} \leq \frac{n-5}{2}$, so 
\[ 2\floor{\log_2 p} + 6 \leq n . \]
By Theorem \ref{thm:Brylawski}, there is a matroid $N$ on at most $2\floor{\log_2 p} + 6$ elements with $c(N)=p$.  
Add to $N$ as many loops as necessary to obtain a matroid $M$ on exactly $n$ elements with $c(M)=p$.  
\end{proof}

\section*{Funding}

This work was supported by a Rutherford Discovery Fellowship.

\section*{Acknowledgement}

We express our thanks to Gary Gordon for pointing out the results of Brylawski that enabled our lower bound, and to Gordon Royle for calculating $f(8)$. 
We also wish to express our thanks to the referee for their careful reading and helpful comments. 

\bibliography{EffR} 
\bibliographystyle{plain} 

\end{document}